\documentclass{amsart}%

\usepackage{amssymb}
\usepackage[usenames,dvipsnames]{color}
\usepackage{enumerate}
\usepackage[colorlinks=true,linkcolor={Brown},citecolor={Brown},urlcolor={Brown}]{hyperref}
\usepackage{cleveref}

\usepackage[all]{xy}
\SelectTips{cm}{}

\swapnumbers

\newtheorem{Thm}[equation]{Theorem}
\newtheorem{Prop}[equation]{Proposition}

\theoremstyle{remark}
\newtheorem{Def}[equation]{Definition}
\newtheorem{Rem}[equation]{Remark}

\newcommand{\nc}{\newcommand}
\nc{\dmo}{\DeclareMathOperator}

\dmo{\coker}{coker}
\dmo{\cone}{cone}
\dmo{\Ind}{Ind}
\dmo{\modname}{mod}%
\dmo{\Res}{Res}

\nc{\ie}{{\sl i.e.,}\ }
\nc{\kk}{\Bbbk}
\nc{\mmod}{\text{-}\!\modname}%
\nc{\onto}{\twoheadrightarrow}%
\nc{\sbull}{{\scriptscriptstyle\bullet}}%

\renewcommand{\le}{\leqslant}
\renewcommand{\ge}{\geqslant}

\begin{document}


\title{Resolutions by permutation modules}
\author{Paul Balmer}
\author{Dave Benson}
\date{2020 March 9}

\address{Paul Balmer, Mathematics Department, UCLA, Los Angeles, CA 90095-1555, USA}
\email{balmer@math.ucla.edu}
\urladdr{http://www.math.ucla.edu/$\sim$balmer}

\address{Dave Benson, Institute of Mathematics, University of Aberdeen, King's College, Aberdeen AB24 3UE, Scotland U.K.}
\email{d.j.benson@abdn.ac.uk}
\urladdr{https://homepages.abdn.ac.uk/d.j.benson/pages/}

\begin{abstract}
We prove that, up to adding a complement, every modular representation of a finite group admits a finite resolution by permutation modules.
\end{abstract}

\subjclass[2010]{}

\thanks{First-named author supported by NSF grant~DMS-1901696.}

\maketitle


Let $G$ be a finite group and $\kk$ be a field of
characteristic~$p>0$ dividing the order of~$G$.
It is well-known that if $G$ has non-cyclic Sylow $p$-subgroups,
the $\kk$-linear representation theory of~$G$ is complicated.
In particular, the Krull--Schmidt abelian category,
$\kk G\mmod$, of finite-dimensional $\kk G$-modules
admits \emph{infinitely many} isomorphism classes of
indecomposable objects. On the other hand, there is a
much simpler class of $\kk G$-modules, the
\emph{permutation modules}, \ie those isomorphic to
$\kk X$ for $X$ a finite $G$-set. The \emph{finite} collection
$\{\kk(G/H)\}_{H\le G}$ additively generates all such modules.

For a $\kk G$-module $M\in\kk G\mmod$, we want to
analyze the existence of what we'll call a
\emph{permutation resolution} for short, \ie an exact sequence
\begin{equation}
\label{eq:resol-intro}%
0\to P_n\to P_{n-1}\to \cdots \to P_1 \to P_0 \to M\to 0
\end{equation}
where all~$P_i$ are permutation modules. Up to direct summands, it is always possible:
\begin{Thm}
\label{thm:resol-intro}%
Let $G$ be a finite group and $M\in\kk G\mmod$.
Then there exists a $\kk G$-module~$N$ such that
$M\oplus N$ admits a finite resolution~\eqref{eq:resol-intro}
by permutation modules.
\end{Thm}

The related problem of  resolutions~\eqref{eq:resol-intro} that are
not only exact but remain exact under all fixed-point functors
has been recently discussed in~\cite{BoucStancuWebb17}.
Allowing $p$-permutation modules $P_i$ (that is, direct
summands of permutation modules), Bouc--Stancu--Webb
prove that such resolutions exist for all~$M$ if and only if~$G$
has a Sylow subgroup that is either cyclic or dihedral (for $p=2$).

Unsurprisingly, \Cref{thm:resol-intro} reduces to a Sylow subgroup
$S$ of~$G$, since every $M$ is a direct summand
of~$\Ind_S^G\Res^G_S(M)$ and since the functor $\Ind_S^G$ is
exact and preserves permutation modules. So we focus on the case
where $G$ is a $p$-group.

For the proof, we shall consider a stronger property:
\begin{Def}
\label{def:good-resol}%
We say that a resolution~\eqref{eq:resol-intro} is
\emph{free up to degree~$m\ge 0$} if $P_i$ is a free module
for $i=0,\ldots,m$. We say that $M$ admits
\emph{good permutation resolutions} if for every integer $m\ge 0$,
there exists a finite resolution~\eqref{eq:resol-intro} by permutation
modules that is free up to degree~$m$.
\end{Def}

\begin{Rem}
\label{rem:good-resol}%
Let $G$ be a $p$-group. A $\kk G$-module $M$ admits good
permutation resolutions if and only if for all~$m\ge 1$ the
$m$th Heller loop $\Omega^mM$ admits a finite permutation
resolution. Also, if $Q$ is free and $M\oplus Q$ admits a
permutation resolution as in~\eqref{eq:resol-intro} then the
epimorphism $P_0\onto M\oplus Q \onto Q$ forces $Q$ to
be a direct summand of~$P_0$ and one can remove
$0\to Q\xrightarrow{=} Q \to 0$ from the resolution.
So if $M\oplus Q$ has a permutation resolution that is free
up to degree~$m$ then so does~$M$.
\end{Rem}

An advantage of good permutation resolutions is the \emph{two out of three property}:
\begin{Prop}
\label{prop:2-out-of-3}%
Let $G$ be a $p$-group. Let $0\to L\to M\to N\to 0$ be an
exact sequence of $\kk G$-modules. If two out of $L$, $M$
and $N$ have good permutation resolutions then so does the third.
\end{Prop}

\begin{proof}
If $P\onto N$ is a projective cover, we obtain by
`rotation' an exact sequence
$0\to \Omega^1N\to L\oplus P\to M\to 0$. In view of
\Cref{rem:good-resol}, we can rotate in this way and reduce to
the case where $L$ and~$M$ admit good permutation resolutions
and then prove that~$N$ does. Let $m\ge 0$. Choose
$P_\sbull\to M$ a permutation resolution of~$M$ that is free up to
degree~$m$. Let $\ell\ge m$ be such that $P_i=0$ for all $i>\ell$.
Now choose $Q_\sbull\to L$ a permutation resolution of~$L$ that
is free up to degree~$\ell$. We have the following picture
(plain part) with exact rows:
\begin{equation}
\label{eq:resol-qi}%
\vcenter{\xymatrix@C=1em@R=1em{
0\ar[r]
& Q_n \ar[r] \ar@{..>}[d]
& \cdots  \ar[r]
& Q_{\ell+1} \ar[r] \ar@{..>}[d]
& Q_\ell \ar[r] \ar@{..>}[d]
& \cdots  \ar[r]
& Q_0  \ar[r] \ar@{..>}[d]
& L \ar[r] \ar[d]
& 0
\\
0\ar[r]
& 0 \ar[r]
& \cdots  \ar[r]
& 0 \ar[r]
& P_\ell \ar[r]
& \cdots \ar[r]
& P_0  \ar[r]
& M \ar[r]
& 0
}}
\end{equation}
The standard lifting argument, using that $Q_j$ is projective for
$j=0,\ldots, \ell$ shows that there exists a lift
$f_\sbull\colon Q_\sbull\to P_\sbull$ of the morphism $L\to M$.
Then the mapping cone complex $\cone(f_\sbull)$ yields a
resolution of $\coker(L\to M)=N$ and this complex~$\cone(f_\sbull)$
has free objects in degree~$0,\ldots,m$ since $P_\sbull$ and~$Q_\sbull$ do.
\end{proof}

Let us discuss an example of \Cref{thm:resol-intro}, where we can
even take $N=0$.
\begin{Prop}
\label{prop:abelem}%
Let $E=(C_p)^{\times r}=C_p\times\cdots\times C_p$ be an
elementary abelian group of rank~$r$. Then every $\kk E$-module
admits good permutation resolutions.
\end{Prop}

\begin{proof}
Consider for each $1\le i\le r$ the (`coordinate-wise') subgroup
\[ H_i=C_p\times \cdots \times C_p\times 1\times C_p\times\cdots\times C_p \]
of rank~$r-1$. Let $m\ge 0$. Inflating from $E/H_i\simeq C_p$ the
usual $2$-periodic resolutions
$0\to \kk\to \kk C_p\to \cdots \to\kk C_p\to \kk\to 0$
of length at least~$m$, we obtain quasi-isomorphisms of
$\kk E$-modules $s_i\colon Q(i)\to \kk[0]$ where the $Q(i)$
are defined as follows:
\[
\vcenter{\xymatrix@R=1em@C=1em{
Q(i):= \ar@<-.8em>[d]_-{s_i}
&& 0 \ar[r]
& \kk \ar[r] \ar[d]
& \kk(E/H_i) \ar[r] \ar[d]
& \cdots \ar[r]
& \kk(E/H_i) \ar[r] \ar[d]
& \kk(E/H_i)  \ar[r] \ar[d]
& 0
\\
\kk[0]=
&& 0\ar[r]
& 0 \ar[r]
& 0 \ar[r]
& \cdots \ar[r]
& 0 \ar[r]
& \kk \ar[r]
& 0
}}
\]
Tensoring all the above, we obtain a quasi-isomorphism
\[ s_1\otimes \cdots \otimes s_r\colon P_\sbull:=
Q(1)\otimes \cdots \otimes Q(r)\to (\kk[0])^{\otimes{r}}\cong
\kk[0], \]
\ie a permutation resolution $P_\sbull$ of~$\kk$. In other words,
we performed an `external tensor' of all the periodic resolutions
over each copy of~$C_p$ in~$E$. Since the Mackey formula gives
by induction $\kk(E/H_{i_1})\otimes \cdots \otimes \kk(E/H_{i_n})\cong
\kk(E/(H_{i_1}\cap \cdots \cap H_{i_n}))$, we have produced a permutation
resolution $P_\sbull$ of~$\kk$ that is easily seen to be free up to
degree~$m$. As $m\ge 0$ was arbitrary, we proved that the trivial
module~$\kk$ admits good permutation resolutions. A general
module $M\in \kk E\mmod$ admits a filtration whose successive
quotients are trivial. We therefore conclude by induction,
via \Cref{prop:2-out-of-3}.
\end{proof}

\begin{Rem}
The proof of \Cref{prop:abelem} shows that the stabilisers in the
permutation resolution may be taken to be products of subsets with
respect to the given decomposition of $E$. Applying the proposition to
a module and its dual shows that given a module $M$ we may form a finite
exact complex of permutation modules with these stabilisers in such
a way that the image of one of the maps is~$M$. This should be compared
with the main theorem of~\cite{BensonCarlson} which shows that
a finite exact sequence of permutation $E$-modules in which the set
of stabilisers has no containment of index $p$ necessarily splits, so that
the image of every map is again a permutation module.
\end{Rem}

\begin{proof}[Proof of \Cref{thm:resol-intro}]
As already mentioned, we can reduce to the case where $G$ is a $p$-group. By \cite{Carlson00}, we know that for every $\kk G$-modules~$M$, there exists a $\kk G$-module~$N$ and a finite filtration $0=L_0\subset L_1\subset \cdots \subset L_s=M\oplus N$ such that every $L_i/L_{i-1}$ is induced from some elementary abelian subgroup~$E_i\le G$. Since the result holds for elementary abelian groups (\Cref{prop:abelem}) and is stable by induction, we see that all $L_i/L_{i-1}$ admit good permutation resolutions. By \Cref{prop:2-out-of-3}, we conclude that so does $M\oplus N$. In particular, $M\oplus N$ has a permutation resolution.
\end{proof}

\medbreak
\noindent
\textbf{Acknowledgements:} The authors are grateful to Serge Bouc, Martin Gallauer and Peter Webb for useful discussions. The authors would like to thank the Isaac Newton Institute for Mathematical Sciences, Cambridge, for support and hospitality during the programmes `K-theory, algebraic cycles and motivic homotopy theory' and `Groups, representations and applications: new perspectives', where work on this paper was undertaken. The Isaac Newton Institute is supported by EPSRC grant no EP/R014604/1.



\end{document}